\documentclass[a4paper,11pt]{amsart}

\usepackage{amsfonts,latexsym,rawfonts,amsmath,amssymb,amsthm, mathrsfs, lscape}
\usepackage[all]{xy}
\usepackage[english]{babel}
\usepackage[utf8]{inputenc}

\usepackage{xcolor}

\usepackage{tikz-cd}

\usepackage{array, tabularx}

\usepackage{setspace}
\usepackage{textcomp}

\usepackage[hypertexnames=false,
backref=page,
    pdftex,
    pdfpagemode=UseNone,
    breaklinks=true,
    extension=pdf,
    colorlinks=true,
    linkcolor=blue,
    citecolor=blue,
    urlcolor=blue,
]{hyperref}

\usepackage[top=1.5in, bottom=.9in, left=0.9in, right=0.9in]{geometry}

\newcommand{\referenza}{}

\newtheorem{thm}{Theorem}[section]
\newtheorem*{thm*}{Theorem \referenza}
\newtheorem{cor}[thm]{Corollary}
\newtheorem*{cor*}{Corollary \referenza}
\newtheorem{lem}[thm]{Lemma}
\newtheorem*{lem*}{Lemma \referenza}

\newtheorem{prop}[thm]{Proposition}
\newtheorem*{prop*}{Proposition \referenza}

\newtheorem*{conj*}{Conjecture \referenza}
\newtheorem{rmk}[thm]{Remark}
\newtheorem*{rmk*}{Remark}

\numberwithin{equation}{section}

\def \R {\mathbb R}
\def \S {\mathbb S}
\def \C {\mathbb C}

\def \P {\mathbb P}

\def \g {\mathfrak g}
\def \l {\mathfrak l}
\def \h {\mathfrak h}
\def \m {\mathfrak m}
\def \n {\mathfrak n}
\def \t {\mathfrak t}

\def \p {\partial}

\DeclareMathOperator\ad{ad}

\renewcommand{\p}{\partial}
\renewcommand{\bar}{\overline}

\allowdisplaybreaks[1]

\title[Bismut--Yamabe problem and Calabi--Yau with Torsion metrics]{On the curvature of the Bismut connection: Bismut--Yamabe problem and Calabi--Yau with Torsion metrics}

\author{Giuseppe Barbaro}

\address{Dipartimento di Matematica ``Guido Castelnuovo", Università la Sapienza, Piazzale Aldo Moro, 5, 00185 Roma, Italy} 
\email{g.barbaro@uniroma1.it}
\thanks{The author is supported by project PRIN2017 ``Real and Complex Manifolds: Topology, Geometry and holomorphic dynamics'' (code 2017JZ2SW5), and by GNSAGA of INdAM}

\begin{document}

\begin{abstract}
    We study two natural problems concerning the scalar and the Ricci curvatures of the Bismut connection. 
    Firstly, we study an analog of the Yamabe problem for Hermitian manifolds related to the Bismut scalar curvature, proving that, fixed a conformal Hermitian structure on a compact complex manifold, there exists a metric with constant Bismut scalar curvature in that class when the expected constant scalar curvature is non-negative. 
    A similar result is given in the general case of Gauduchon connections. 
    We then study an Einstein-type condition for the Bismut Ricci curvature tensor on principal bundles over Hermitian manifolds with complex tori as fibers. 
    Thanks to this analysis we construct explicit examples of Calabi--Yau with torsion Hermitian structures and prove a uniqueness result for them.
\end{abstract}

\maketitle

\section{Introduction}Given a Hermitian manifold $(M,J,g)$, there are several connections which are compatible with both the metric $g$ and the complex structure $J$, meaning that they leave them parallel. Among these, the Bismut connection is the only one which has skew-symmetric torsion. Thanks to this property it takes on great interest in String Theory, see the work of Ivanov and Papadopoulos \cite{Ivanov Papa}. It also has applications in Differential Geometry, see for example \cite{Streets}, and recently, Garcia-Streets and Streets-Tian showed some interesting links with the Generalized Complex Geometry, see \cite{Garcia Streets} and \cite{Streets Tian 2}. 

In view of having a better understanding of the geometry of the Bismut connection, in this note we study the classical problems of constant scalar curvature and constant Ricci curvature (in the sense that the Ricci tensor is a multiple of the metric) with a focus on the case Ricci flat curvature. In the first part of this note, we adapt the techniques used in \cite{Chern-Yamabe} for the Chern-Yamabe problem to the general case of Gauduchon connections $\nabla^t$, which are an affine line of Hermitian connections including the Chern connection $\nabla^{Ch}$ (for $t=1$) and the Bismut connections $\nabla^+$ (for $t=-1$). In this way, in Theorems \ref{thm: zero gau-yam} $\&$ \ref{thm:  gau-yam} we could solve the \textit{Gauduchon-Yamabe} problem of finding a constant $\nabla^t$-scalar curvature metric in a given conformal class if $\Gamma_M^{t}(\{\omega\})\leq 0$ or $\Gamma_M^{t}(\{\omega\})\geq 0$ depending on $t>\frac{1}{1-n}$ or $t<\frac{1}{1-n}$, where $n$ is the complex dimension of the manifold. Here, $\Gamma_M^{t}(\{\omega\})$ is the Gauduchon degree with respect to $\nabla^t$ associated to the conformal class $\{\omega\}$, which is defined as
$$ \Gamma_M^{t}(\{\omega\}) := \int_M S^{t}(\eta) d\mu_\eta \;, $$ 
where $\eta$ is the unique volume-one Gauduchon representative of $\{\omega\}$ and $S^t(\eta)$ is the $\nabla^t$-scalar curvature associated to $\eta$. The Gauduchon degree with respect to the Chern connection corresponds to the degree of the anti-canonical line bundle $K_M^{-1}$ (studied in \cite{gauduchon-mathann}), while for $t=-1$ we get the Gauduchon degree of $\{\omega\}$ with respect to the Bismut connection, which we denote by $\Gamma^+_M(\{\omega\})$. 
Theorems \ref{thm: zero gau-yam} $\&$ \ref{thm:  gau-yam} extend the results of \cite{Chern-Yamabe} about the Chern connection and we have the following theorem as a particular case of them.
\begin{thm*}[Corollary \ref{cor: bis-yam}]
    Let $M$ be a compact complex manifold with $\dim_\C M\geq 3$ and Hermitian structure $(\omega,J)$.
    If $\Gamma_M^+(\{\omega\}) \geq 0$, then, up to scaling, there exists a unique  $\widetilde{\omega}\in\{\omega\}$ with constant Bismut scalar curvature. Moreover, its Bismut scalar curvature satisfies  $S^{+}(\widetilde{\omega})=\Gamma_M^+(\{\omega\})$.
\end{thm*}

The second part of the note is devoted to the study of metrics with constant Bismut Ricci curvature on the total spaces of rank one toric bundles ($T^2$-bundles) over Hermitian manifolds. 
$T^2$-fibrations over Calabi--Yau surfaces were used by Fu and Yau to show explicit solutions to the Hull--Strominger system, see \cite{Fu Yau} and \cite{Fu Yau 2}, while in \cite{Gran Gran Poon} the authors studied the so-called Calabi--Yau with torsion (CYT) condition on the total spaces of toric bundles over K{\"a}hler manifolds. 
Here, we focus on the vanishing of the Bismut Ricci tensor, describing the CYT condition for metrics that satisfy a natural ansatz on this class of manifolds. 
We then briefly analyze the Einstein-type equation $(Ric^+\omega)^{1,1}=\lambda\omega$ with $\lambda\in\R$ in the same setting. 

The CYT manifolds play a role in Physics after the works of Strominger \cite{Strominger} and Hull \cite{Hull}. Moreover, there is an interest in finding explicit examples of pluriclosed Hermitian structures (meaning that the K{\"a}hler $2$-form associated to the metric is $dd^c$-closed) which also satisfy $(Ric^+\omega)^{1,1}=\lambda\omega$ cause they are static points of the {\em pluriclosed flow} of Streets and Tian.
In Section \ref{sec: Bismut Einstein} we analyze these structures on the Calabi--Eckmann manifolds $\mathbb{S}^{2n+1}\times\mathbb{S}^{2m+1}$ (with $n,m\geq 0$). 
Indeed, the simplest examples of {\em Bismut flat} manifolds are given by the Hopf surface $\S^1\times\S^3$ and the Calabi--Eckmann threefold $\S^3\times\S^3$, and in \cite{Garcia Streets} the authors asked if it was possible to construct other special Hermitian structures on the Calabi--Eckmann manifolds of higher dimension.
The existence of CYT structures on them can be obtained by applying Theorem 3 of \cite{Gran}, while we prove a result about the uniqueness.
\begin{thm*}[Corollary \ref{cor: CYT on CE}]
    Given a Calabi--Eckmann manifold $M_{n,m}:=\mathbb{S}^{2n+1}\times\mathbb{S}^{2m+1}$, the standard Hermitian structure $(J,g)$ gives a CYT structure on it. Moreover, $g$ is the only homogeneous CYT structure on $(M_{n,m},J)$.
\end{thm*}

The above theorem comes as a special case of a more general result that has an analogous statement on Class $\mathcal{C}$ manifolds (defined in \cite{Podesta}).
These are the total spaces of homogeneous principal $T^2$-bundles over the product of two compact irreducible Hermitian symmetric spaces.
\begin{thm*}[Theorem \ref{thm: CYT class C} \& Theorem \ref{thm: CYT uniqueness}]
    Take a class $\mathcal{C}$ manifold as in \cite{Podesta}, that is a product $M_1\times M_2$ fibering over two generalized flag manifolds $X_1=G_1/H_1$ and $X_2=G_2/H_2$ with $\S^1$-fibers, and equip it with a standard complex structure. Then, there exists a CYT metric on it. 
    Moreover, if none of the $X_i$'s is $\mbox{SO}(k+2)/\mbox{SO}(2)\times\mbox{SO}(k)$ for $k\geq 3$, then this metric is the unique (up to homothety) CYT metric among the homogeneous ones.
\end{thm*}
The interest in this kind of result comes from the fact that a compact simply connected homogeneous manifold $G/H$ with an invariant complex structure $J$ (called C-space in \cite{Wang}) is K{\"a}hlerian if and only if it is a generalized flag manifold, namely when $G$ is a semisimple Lie group and $H$ is the centralizer of a torus in $G$, as stated in \cite{Borel}. 
In such case, in \cite{Matsushima} is proved that they can be endowed with a (unique) invariant K{\"a}hler--Einstein metric, while there is a general interest in finding special invariant metrics on the non-K{\"a}hler C-spaces $G/H$. 
Theorem 3 in \cite{Gran} proves the existence of a CYT structure on compact simply connected homogeneous manifolds $G/H$ with a G-invariant complex structure of vanishing first Chern class.
However, Corollary \ref{cor: lots of CYT} ensures that, in general, it is not unique.
On the other hand, the standard Hermitian structures on class $\mathcal{C}$ manifolds are the unique invariant CYT structure on them.

\section{Preliminaries and notation}\label{sec:prelim}
In all the following sections $(M,g,J)$ will be a Hermitian manifold of complex dimension $n\geq 2$, and $\omega:=g(J\cdot,\cdot)$ will denote its associated K{\"a}hler $(1,1)$-form. In particular, in local holomorphic coordinates $\{z_i\}_i$,
$$\omega= \sqrt{-1} g_{i\bar j}dz^i\wedge d\bar z^j \;,$$
where $g_{i\bar j}=g\left(\frac{\p}{\p z_i}, \frac{\p}{\p \bar z_j}\right)$. Moreover, $(g^{i\bar j})_{i,j}$ will henceforth denote the inverse of $(g_{i\bar j})_{i,j}$.

\subsection{Gauduchon connections} For $t\in\R$ the \emph{Gauduchon connections} $\nabla^{t}$ associated to $(g,J)$ are Hermitian connections on $M$ with prescribed torsion, where by Hermitian connections we mean connections which are compatible with both the metric and the complex structure, i.e. $\nabla g = \nabla J =0$. They are described with respect to the Levi--Civita connection $\nabla^{LC}$ as, 
\begin{equation*}
    g(\nabla^{t}_{x}y,z)=g(\nabla^{LC}_{x}y,z)+\frac{1-t}{4}Jd\omega(x,y,z)+\frac{1+t}{4}d\omega(Jx,y,z) \;,
\end{equation*}
where  $J$ acts as $Jd\omega(\cdot,\cdot,\cdot)=-d\omega(J\cdot,J\cdot,J\cdot)$. Since the Levi--Civita connection is torsion-free, the above formula is prescribing the torsion of these connections as 
$$ T^t(x,y,z)=\frac{1-t}{2}Jd\omega(x,y,z)+\frac{1+t}{4}\left( d\omega(Jx,y,z) + d\omega(x,Jy,z) \right) \;. $$
In particular, for $t=1$ we recover the Chern connection, while for $t=-1$ we get the Bismut connection, which thus has torsion equal to $d^c\omega$. We remark that the Bismut connection is the unique Hermitian connection with totally skew-symmetric torsion. We will henceforth indicate the Bismut connection as $\nabla^+$, and any of its curvature tensors with the superscript $+$.
The Christoffel symbols of the Gauduchon connection can be easily computed and are:
\begin{align}
    \left(\Gamma^t\right)_{ij}^{k}&=g^{k\overline{s}}\left(\frac{1+t}{2}\partial_{i}g_{j\overline{s}}+\frac{1-t}{2}\partial_{j}g_{i\overline{s}}\right) \nonumber\\
    \left(\Gamma^t\right)_{\overline{i}j}^{k}&=\frac{1-t}{2}g^{k\overline{s}}\left(\overline{\partial}_{i}g_{j\overline{s}}-\overline{\partial}_{s}g_{j\overline{i}}\right)\\
    \left(\Gamma^t\right)_{i\overline{j}}^{k}&=0\nonumber
\end{align}
At a fixed point $p\in M$ we can choose special holomorphic coordinates $\{z_i\}$ such that $g_{i\bar j}(p)=\delta_{ij}$ and the Christoffel symbols of the Levi--Civita connection vanish at $p$, i.e. $(\Gamma^{LC})_{ij}^k(p)=0$. With these coordinates, we compute the curvature tensors of the Gauduchon connections as
\begin{align}\label{eq: Gauduchon curature tensor}
    R^t_{i\bar j k \bar l}(g)&= -\delta_{p\bar l} \left(\frac{\p
    }{\p \bar z^j}\left(\Gamma^t\right)^{p}_{ik} - \frac{\p }{\p z^i}\left(\Gamma^t\right)^{p}_{\bar jk} + \left(\Gamma^t\right)_{ik}^{s}\left(\Gamma^t\right)^{p}_{\bar js} - \left(\Gamma^t\right)_{ \bar jk}^{s}\left(\Gamma^t\right)^{p}_{is} \right) \nonumber\\
    &=\frac{1-t}{2}\left(\frac{\p^2 g_{k\bar l}}{\p z_i \p\bar z_j} - \frac{\p^2 g_{k\bar j}}{\p z_i \p\bar z_l} - \frac{\p^2 g_{i\bar l}}{\p z_k \p\bar z_j} \right) - \frac{1+t}{2} \frac{\p^2 g_{k\bar l}}{\p z_i \p\bar z_j} + \sum_q (1-t)^2\frac{\p g_{q\bar l}}{\p z_i}\frac{\p g_{k\bar j}}{\p\bar z_q} - t^2 \frac{\p g_{i\bar q}}{\p z_k}\frac{\p g_{q\bar l}}{\p\bar z_j}
\end{align}
We define the Ricci tensor associated to the Gauduchon connection $\nabla^t$ as the contraction of the endomorphism part of its curvature tensor; hence, in local coordinates,
$$ Ric^t_{\cdot \cdot}(g) = g^{k\bar l} R^t_{\cdot \cdot k\bar l}(g) \;.$$
Contracting again we obtain the $\nabla^t$-scalar curvature 
$$ S^t(g)=g^{i\bar j}Ric^t_{i\bar j}(g) \;.$$

From a direct computation using \eqref{eq: Gauduchon curature tensor}, we obtain the following useful formula, which is well known for the Bismut connection. 
\begin{prop}\label{prop: relation Bismu Chern Ricci}
     Let $(M,g,J)$ be a Hermitian manifold of complex dimension $n$, and let $\omega$ be its associated K{\"a}hler $(1,1)$-form. For $t\in\R$, the Ricci curvature form of the Gauduchon connection $\nabla^{t}$ associated to $(g,J)$ is given by the formula: 
     $$ Ric^t(g)=\frac{t-1}{2}dd_g^*\omega - \sqrt{-1} \partial\overline{\partial} \log\omega^n \;.$$
     In particular, the (1,1)-component of the Gauduchon Ricci curvature form satisfies
     \begin{equation}\label{eq: ricci t and chern}
         (Ric^t(g))^{1,1}=\frac{t-1}{2}(\partial\partial^*_g \omega +\overline{\partial\partial}^*_g \omega) - \sqrt{-1} \partial\overline{\partial} \log\omega^n \;.
     \end{equation}
\end{prop}
In the above proposition, $d^*_g=\p^*_g+\bar\p^*_g$ where $\partial^*_g:\land^{p+1,q}M\rightarrow\land^{p,q}M$ and $\overline{\partial}^*_g:\land^{p,q+1}M\rightarrow\land^{p,q}M$ are the $L^2_g$-adjoint operators of $\partial$ and $\overline{\partial}$ respectively.
We recall the local formulas for these operators, which can be found, for example, in \cite{Streets Tian 3}. In local holomorphic coordinates, 
\begin{align*}
    \left(\p^*_{g} \omega \right)_{\bar k} =& \sqrt{-1} g^{p \bar q}
    \left(\p_{\bar q} g_{p \bar k} - \p_{\bar k} g_{p \bar q} \right)\\
    \left( \bar\p _g^* \omega \right)_j =& \sqrt{-1} g^{p \bar q} \left( \p_p g_{j \bar q} - \p_j g_{p \bar q} \right)
\end{align*}
We also fix here the notation for the trace of a two form $\alpha$ with respect to $\omega$: $$tr_\omega \alpha := \sqrt{-1}g^{i\overline{j}}\alpha_{i\overline{j}}\;,$$ where we used local holomorphic coordinates $\{z_i\}_i$.\\
Thanks to Proposition \ref{prop: relation Bismu Chern Ricci} we can describe the case where two Gauduchon scalar curvatures with different Gauduchon parameters coincide.
\begin{prop}\label{prop: S1=S2 - balanced}
    Let $(M,g,J)$ be a compact Hermitian manifold and take two Gauduchon parameters $t_1\neq t_2$. Then the following conditions are equivalent: 
    \begin{enumerate}
        \item[i.] $Ric^{t_1}(\omega)=Ric^{t_2}(\omega)$;
        \item[ii.] $S^{t_1}(\omega)=S^{t_2}(\omega)$;
        \item[iii.] $g$ is balanced.
    \end{enumerate}
\end{prop}
\begin{proof}
    Obviously, $(i)\Rightarrow (ii)$. 
    As for $(ii)\Rightarrow(iii)$, taking the trace in~\eqref{eq: ricci t and chern} we have that $s^{t_1}=s^{t_2}$ if and only if 
    $$ tr_\omega (\p\p^*_g \omega + \bar{\p\p}^*_g \omega)=0 \;. $$
    However, integrating over $M$ we get that
    $$ \int_M tr_\omega (\p\p^*_g \omega) = (\p\p^*_g \omega,\omega)_g = (\p^*_g \omega,\p^*_g \omega)_g = |\p^*_g \omega |^2_g \;,$$
    and similarly for $tr_\omega (\bar{\p\p}^*_g \omega)$.
    Thus both $\p^*_g \omega$ and $\bar{\p}^*_g \omega$ vanish, which means that $\theta=Jd^*\omega=0$ and $g$ is balanced. 
    Finally, going backward through this argument yields $(iii)\Rightarrow (i)$.
\end{proof}

\medskip

We compute the variation of the $\nabla^{t}$-Ricci curvature under a conformal change, which easily comes from the above formulas:
\begin{equation}\label{eq: conformal Ricci Gauduchon}
    (Ric^t(e^f g))^{1,1}=(Ric^t(g))^{1,1} + (t-nt-1)\sqrt{-1}\partial\overline{\partial}f\;,
\end{equation}
while the $(2,0)$ and the $(0,2)$ components do not change.
Taking the trace, we obtain 
$$ S^t(e^f g) = e^{-f} \left(S^t(g) + \left(1+nt-t \right)\sqrt{-1}tr_\omega\partial\overline{\partial} f\right) \;.$$
    
\subsection{Chern Laplacian}
We recall the definition of the Chern Laplacian $\Delta^{Ch}_{\omega}$ associated to the Hermitian metric $\omega$ on a smooth function $f$ as
$$ \Delta^{Ch}_\omega f = 2\sqrt{-1}\mathrm{tr}_\omega\overline\partial\partial f \;, $$
or, in local holomorphic coordinates $\{z_i\}_i$ as
$$ \Delta^{Ch}_\omega \stackrel{\text{loc}}{=} -2g^{i\bar j}\partial_i \partial_{\bar j} \;.$$ 

\begin{rmk}
    With this notation, the variation formula for the $\nabla^{t}$-scalar curvature under a conformal change becomes:
    \begin{equation}\label{Eq: Scalar_curvature_conformal_change}
         S^t(e^f g) = e^{-f} \left(S^t(g) + \frac{1}{2}(1+nt-t)\Delta^{Ch}_g f\right) \;.
    \end{equation}
\end{rmk}

In \cite{gauduchon-mathann}, Gauduchon made explicit the relation between the Hodge--de Rham Laplacian $\Delta_{d,\,\omega}$ and the Chern Laplacian $\Delta^{Ch}_\omega$ on smooth functions through the torsion $1$-form, which we recall is defined by the equation
$$ d\omega^{n-1} = \theta\wedge\omega^{n-1} \;.$$

\begin{lem}[\cite{gauduchon-mathann}, pages 502-503]\label{lem:chern-laplacian}
 Let $M$ be a compact complex manifold endowed with a Hermitian metric $\omega$ with torsion 1-form $\theta$. The Chern Laplacian on smooth functions $f$ has the form
 $$ \Delta^{Ch}_\omega f = \Delta_d f + (df,\, \theta )_\omega \;. $$
\end{lem}
In particular, the Chern Laplacian is a differential elliptic operator of 2nd
order without terms of order 0 and its index agrees with the index of the Hodge--de Rham Laplacian, as is outlined in \cite{gauduchon-cras1977}. Moreover, the Chern Laplacian and the Hodge--de Rham Laplacian on smooth functions coincide when $\omega$ is balanced (i.e. if $\theta=0$), and $ \Delta^{Ch}_\omega f_{|_{p}} \geq 0$ whenever $f$ is a smooth real function on $M$ which attains a local maximum at $p\in M$.

\subsection{Hermitian conformal structures}
Given a metric $\omega$ on $M$, the Hermitian conformal class of $\omega$ will be denoted by
$$ \{\omega\} := \left\{ \exp({f})\,\omega \;\middle|\; f\in\mathcal{C}^\infty(M;\R) \right\} \;. $$
The following fundamental result by Gauduchon ensures the existence of a Gauduchon metric (i.e. a metric with $d^*\theta=0$) in any Hermitian conformal class.
\begin{thm}[\cite{gauduchon-cras1977}, Théorème 1]
 Let $M$ be a compact complex manifold of complex dimension $\dim_\C M\geq2$, and fix a Hermitian conformal structure $\{\omega\}$.
 Then there exists a unique Gauduchon metric $\eta$ in $\{\omega\}$ such that $\int_M d\mu_\eta=1$.
\end{thm} 
Using this result we can consider the following normalized conformal class
$$ \{\omega\}_1 := \left\{ \exp(f)\eta\in\{\omega\} \;\middle|\; \int_M \exp(f)d\mu_\eta=1 \right\} \subset \{\omega\} \;, $$
where we denoted by $\eta\in\{\omega\}$ the unique Gauduchon representative of volume $1$.\\
With this choice of $\eta$ we can also introduce a natural \emph{invariant} of the conformal class $\{\omega\}$, namely, the \emph{Gauduchon degree} 
$$ \Gamma_M (\{\omega\}) \in \R \;, $$
defined as 
$$ \Gamma_M(\{\omega\}) := \frac{1}{(n-1)!} \int_M c_1^{BC}(K_M^{-1})\wedge\eta^{n-1} = \int_M S^{Ch}(\eta) d\mu_\eta \;. $$
We extend it to any Gauduchon parameter $t\in\R$, 
$$ \Gamma_M^t(\{\omega\}) := \int_M S^{t}(\eta) d\mu_\eta\;. $$
This value is related to the expected constant $\nabla^{t}$-scalar curvature as follows.
\begin{prop}\label{prop:gamma-sign}
    Let $(M,g,J)$ be a compact Hermitian manifold. Assume that $\omega'\in\{\omega\}$ has constant $\nabla^{t}$-scalar curvature equal to $\lambda\in\R$. Then $\omega'\in\{\omega\}_1$ if and only if
    $$ \Gamma_M^t(\{\omega\}) = \lambda \;. $$
    In particular, the sign of the $\nabla^{t}$-scalar curvature of a potential constant $\nabla^{t}$-scalar curvature metric in $\{\omega\}$ agrees with the sign of $\Gamma_M^t(\{\omega\})$.
\end{prop}
\begin{proof}
 Suppose that $e^f\omega\in\{\omega\}_1$ has constant $\nabla^{t}$-scalar curvature $\lambda$. As representative in $\{\omega\}$, fix the unique Gauduchon metric $\eta\in\{\omega\}$ of volume $1$ and denote by $\theta$ its torsion 1-form. Equation \eqref{Eq: Scalar_curvature_conformal_change} yields
 $$ \frac{1}{2}(1+nt-t)\int_M \Delta^{Ch}_\eta f d\mu_\eta + \int _M S^{t}(\eta) d\mu_\eta = \lambda \int_M \exp(f) d\mu_\eta \;, $$
 where
 $$ \int_M \Delta^{Ch}_\eta f d\mu_\eta = \int_M \Delta_d f d\mu_\eta + \int_M (df, \theta) d\mu_\eta = \int_M \Delta_d f d\mu_\eta + \int_M (f, d^*\theta) d\mu_\eta = 0 \;, $$
 since $\eta$ is Gauduchon.
 Therefore
 $$ \Gamma_M^t(\{\omega\}) = \int_M S^{t}(\eta) d\mu_\eta = \lambda \int_M \exp(f) d\mu_\eta = \lambda \;, $$
 yielding the first implication.
 
 \medskip
 
 On the other hand, if we have a metric $e^f\omega\in\{\omega\}$ with constant $\nabla^{t}$-scalar curvature equal to $\Gamma_M^t(\{\omega\})$ we can scale it by a constant $e^c$ so that $e^{f+c}\,\omega$ stays in the normalized conformal class $\{\omega\}_1$ and its Gauduchon scalar curvature becomes $e^{-c}\,\Gamma^t_M(\{\omega\})$. Here $c$ is such that 
 $$ e^{-c}=\int_M e^f d\mu_\eta \;.$$
 Note that $e^{f+c}\,\omega$ is a constant $\nabla^{t}$-scalar curvature metric in $\{\omega\}_1$, hence it has scalar curvature equal to $\Gamma^t_M(\{\omega\})$. Thus, finally, $c=0$. 
\end{proof}

\section{Gauduchon-Yamabe problem}\label{sec:cy}
The Yamabe problem, consisting in finding a constant scalar curvature metric in the conformal class of a given Riemannian metric, is well understood in the Riemannian setting, while in the Hermitian setting, the Chern-Yamabe problem was introduced and studied in \cite{Chern-Yamabe}. In this note we study it for all the Gauduchon connections, i.e. given a Hermitian manifold $(M,g,J)$ we look for a constant $\nabla^t$-scalar curvature metric $\widetilde{\omega}$ in the conformal class $\{\omega \}$. Thanks to the conformal changing equation for the scalar curvature of $\nabla^t$ (\ref{Eq: Scalar_curvature_conformal_change}), this problem reduces to solve a semi-linear elliptic equation of $2$nd order:
\begin{equation}\label{eq: elliptic_equation_G-Y_problem}
    C_t \Delta^{Ch}_\omega f + S^t(\omega) = \lambda\exp(2f) \;,
\end{equation} where $C_t=1+nt-t$ and $\lambda$ is the expected constant scalar curvature value, equal to $\Gamma^t_M(\{\omega\})$ by Proposition \ref{prop:gamma-sign}.\\
Note that for $t=1$ we recover the Chern-Yamabe problem which has been studied in \cite{Chern-Yamabe} by Angella, Calamai and Spotti. In what follows, we extend their arguments to the more general Gauduchon-Yamabe problem, obtaining their results on the Chern connection and our results on the Bismut connection as particular choices of Gauduchon connections in Theorems \ref{thm: zero gau-yam} $\&$ \ref{thm:  gau-yam}. 

\subsection{Linear case}\label{sec:zero-gaud}
In case of $\Gamma^t_M(\{\omega\})=0$, the semi-linear elliptic differential equation \eqref{eq: elliptic_equation_G-Y_problem} becomes just linear since we shall take $\lambda=0$, and so we get a solution for the corresponding Gauduchon-Yamabe problem whenever $C_t\neq 0$.

\begin{thm}\label{thm: zero gau-yam}
 Let $M$ be a compact complex manifold with Hermitian structure $(\omega,J)$. If the Gauduchon parameter $t$ is such that $C_t\neq 0$ and $\Gamma^t_M(\{\omega\}) = 0$, then there exists a unique metric $\widetilde{\omega}\in\{\omega\}_1$ such that it has constant scalar curvature with respect to the $\nabla^t$ Gauduchon connection. Moreover, $S^t(\widetilde{\omega})=\Gamma^t_M(\{\omega\})=0$.
\end{thm}
\begin{proof}
Fix $\eta\in\{\omega\}$ the unique Gauduchon representative in $\{\omega\}$ with volume $1$.
We should solve \eqref{eq: elliptic_equation_G-Y_problem} with $\lambda=0$, that is
$$ C_t\Delta^{Ch}_\eta f = -S^{t}(\eta) \; . $$

Using the relation in Lemma \ref{lem:chern-laplacian} it can be shown (see \cite{gauduchon-mathann}) that the Kernel of the Chern Laplacian consists of just the constant functions. Indeed, we recall that
$$ \Delta^{Ch}_\eta f = \Delta_d f + (df,\, \theta )_{\eta} \;, $$
where $\theta$ denotes the torsion 1-form of $\eta$. Thus if we take a function $u$ in $\ker(\Delta^{Ch})$ we have
\begin{align*}
 0= \int_M u \Delta^{Ch}_\eta u \,d\mu_{\eta} = \int_M \left( |\nabla u|^2 + \frac{1}{2} (du^2 , \, \theta ) \right) d\mu_{\eta}  = 
\int_M  |\nabla u|^2  \,d\mu_{\eta} \; ,
\end{align*}
since $d^*\theta=0$ because $\eta$ is Gauduchon.\\
It follows that two conformal metrics with zero Gauduchon scalar curvature differ by a multiplicative constant, which in turn must be one if they are both in $\{\omega\}_1$, and so we get the uniqueness.

\medskip

From the above equality, it is also possible to compute (see \cite{gauduchon-mathann}) the adjoint of $\Delta^{Ch}_\eta$ on smooth functions $u$ as 
$$ (\Delta^{Ch}_\eta)^* u = \Delta_d u - (du,\, \theta )_{\eta} \, . $$ 
Thus the same computation applies and hence also the Kernel of the adjoint of the Chern Laplacian of a Gauduchon metric consists of just the constants. Since the integral of  $-S^{t} (\eta)$ is zero by hypothesis, $-C_t^{-1}S^{t} (\eta)\in \left( \ker (\Delta^{Ch}_\eta)^* \right)^\perp = \mathrm{im}\, \Delta^{Ch}_\eta$.
We thus achieve the existence of a metric of zero Gauduchon scalar curvature.
\end{proof}

\subsection{Non-linear case}\label{sec: non-linear}

Here we provide a positive answer for the Gauduchon-Yamabe problem when $C_t\Gamma^t_M(\{\omega\})<0$. As a particular case we will obtain the solution of the Bismut--Yamabe problem when the Gauduchon degree $\Gamma^{-1}_M(\{\omega\})$ is strictly positive and the complex dimension of the manifold is greater than three. 

\begin{thm}\label{thm:  gau-yam}
    Let $M$ be a compact complex manifold with Hermitian structure $(\omega,J)$.
    Fix a Gauduchon parameter $t$ for which $C_t\Gamma_M^t(\{\omega\}) < 0$. Then there exists a unique $\widetilde{\omega}\in\{\omega\}_1$ with constant $\nabla^t$-scalar curvature. Moreover, its Gauduchon scalar curvature satisfies $S^{t}(\widetilde{\omega})=\Gamma_M^t(\{\omega\})$.
\end{thm}
\begin{proof}
    Fix $\eta\in\{\omega\}$ the unique Gauduchon representative in $\{\omega\}$ with volume $1$. By hypothesis, we have
    $$ C_t\Gamma_M^t(\{\omega\}) = C_t\int_M S^{t}(\eta) d\mu_\eta < 0 \;. $$
    The proof of the existence of a constant $\nabla^t$-scalar curvature metric consists of two steps. We apply a continuity method to prove the existence of a constant $\nabla^t$-scalar curvature metric in $\{\omega\}$ of class $\mathcal{C}^{2,\alpha}$; then we exploit the structure of the elliptic equation by a standard bootstrap argument to prove that it is smooth.

\medskip

    Before starting with the continuity method, we need a preliminary step. Namely, we prove that in the normalized conformal class $\{\omega\}_1$ there is a metric which has $\nabla^t$-scalar curvature of constant sign $-sign(C_t)$. By this, we can assume that $C_tS^{t}(\omega)<0$ at every point.\\
    Consider the equation
    \begin{align}
        \Delta^{Ch}_\eta f = -S^{t} (\eta) + \int_M S^{t}(\eta) d\mu_{\eta} \: .
    \end{align}
    Since $\eta$ is Gauduchon, arguing as in the proof of Theorem \ref{thm: zero gau-yam}, the above equation has a solution $f\in\mathcal{C}^\infty(M;\R)$, which is unique once we require $\int_M \exp(2f/C_t)d\mu_\eta=1$.
    Then $\exp(2f/C_t)\eta\in\{\omega\}_1$ satisfies
    \begin{align*}
        C_tS^{t} (\exp(2f/C_t)\eta) &=  \exp(-2f/C_t)C_t \left(S^{t}(\eta)+\Delta^{Ch}_\eta f\right) \\[5pt]
        &= \exp(-2f/C_t)C_t \int_M S^{t}(\eta) d\mu_{\eta} \\[5pt]
        &= \exp(-2f/C_t)C_t \Gamma_M^t(\{\omega\}) < 0 \; .
    \end{align*}

\medskip

    Now we can set up the following continuity path using as reference metric in the conformal class of $\eta$ the above metric $\omega$ with $C_tS^{t}(\omega)<0$.
    Consider the map, for $\alpha \in (0,\,1)$,
    \begin{align*}
        \mathrm{GaYa} \colon [0,\,1]\times\mathcal{C}^{2,\alpha}(M ; \mathbb{R}) \rightarrow  C^{0,\alpha} (M ; \mathbb{R}) \;,
    \end{align*}
    such that
    $$ \mathrm{GaYa}(s,f) := \Delta^{Ch}_\omega f + s S^{t} (\omega)  - \lambda \exp(2f/C_t) + \lambda (1-s)\; . $$
    Let us define the set
    $$ S:= \left\{ s\in [0,\, 1] \;\middle|\;  \exists f_s \in C^{2 , \, \alpha} (M ; \mathbb{R}) \mbox{ such that } \mathrm{GaYa}(s,\, f_s)=0 \right\} \; , $$
    which trivially is non-empty since $\mathrm{GaYa}(0,0)=0$. Thus, we should prove that it is also open and closed since the expected solution is achieved when $s=1$. 

\medskip

    We start with the open condition. The implicit function theorem for Hilbert spaces guarantees that $S$ is open as long as the linearization of $\mathrm{GaYa}$ with respect to the second variable is bijective. Hence, we prove that, for a fixed solution $\mathrm{GaYa}(s_0, f_{s_0})=0$, the linearized operator of $\mathrm{GaYa}$,
    $$ D \colon C^{2 , \, \alpha} (M ; \mathbb{R}) \rightarrow C^{0 , \, \alpha} (M ; \mathbb{R}) $$
    defined by
    $$ v \mapsto Dv : = \Delta^{Ch}_\omega v - \lambda \exp\left(2f_{s_0}/C_t\right)\cdot 2v/C_t  $$
    is bijective. Let us remark that $D$ differs from the Chern Laplacian by a compact operator, thus they have the same index, zero. This means that injectivity directly implies surjectivity, hence we are reduced to prove the former.

    If $v$ belongs to $\ker D$, then at a maximum point $p$ for $v$ there holds
    \begin{align}
        - \lambda \exp(2f_{t_0}(p)/C_t)\cdot 2v(p)/C_t \leq 0 \, ,
    \end{align}
    and hence $v(p)\leq 0$, since $-\lambda /C_t>0$. Similarly, at a minimum point $q$ for $v$, there holds $v(q)\geq 0$. Thus, $\ker D = \{0\}$. 

\medskip

    To show that $S$ is also closed we argue as follows. Take $\{ s_n\} \subset S$ a sequence converging to $s_{\infty}$ and $f_{s_n} \in\mathcal{C}^{2,\alpha}(M;\R)$ such that $\mathrm{GaYa} (s_n , \, f_{s_n} ) = 0$ for any $n$; we will use the Ascoli-Arzel\`a theorem to prove that the $f_{s_n}$ converge in $\mathcal{C}^{2,\alpha}(M;\R)$ to a function $f_{\infty}$ such that $\mathrm{GaYa} (s_\infty , \, f_{s_\infty} ) = 0$. To use that theorem, we first need uniform $L^\infty$ estimates of the solutions $f_{s_n}$.
    \begin{lem}\label{lem: C zero unif estimate}
        There exists a positive constant $K$, depending only on $M$, $\omega$, $\lambda$ and $t$ such that, for any $n$, we have
        \begin{align}
            \|f_{s_n}\|_{L^\infty} \leq K \; .
        \end{align}
    \end{lem}
    \begin{proof}
        By hypothesis the functions $f_{s_n}$ satisfy $\mathrm{GaYa} (s_n , \, f_{s_n} ) = 0$, which means that the following equality holds:
        \begin{align}\label{appia}
            \Delta^{Ch}_\omega f_{s_n} + s_n S^{t} (\omega)  - \lambda \exp\left(2f_{s_n}/C_t \right) + \lambda (1-s_n) = 0 \, .
        \end{align}
        We distinguish the two cases $C_t>0$ and $C_t<0$ which correspond to $\lambda<0$ or $\lambda>0$ respectively. We also recall that, by the preliminary step in the proof, $S^{t} (\omega)$ can be supposed to be a negative function when $C_t>0$ and positive when $C_t<0$.
        
        Suppose $C_t>0$ and take a maximum point $p$ for $f_{s_n}$. Then, at $p$, there holds
        \begin{align}
            - \lambda \exp\left(2f_{s_n}(p)/C_t \right) \leq -s_n S^{t} (\omega)(p) - \lambda (1-s_n) \leq -\left( \min_M S^{t} (\omega) \right) -\lambda \, .
        \end{align}
        On the other hand, at a minimum point for $f_{s_n}$, say $q$, there holds
        \begin{eqnarray*}
            - \lambda \exp\left(2f_{s_n}(q)/C_t \right) &\geq& -s_n S^{t} (\omega)(q) - \lambda (1-s_n) \geq s_n (-S^{t}(\omega)(q) +\lambda) -\lambda\\[5pt]
            &\geq& \min \left\{ \min_M \left(-S^{t}(\omega)\right) ,\, -\lambda  \right\} >0\; .
        \end{eqnarray*}
        The above estimates provide the claimed uniform constant $K_0$. The same argument holds also for $C_t<0$, indeed, in this case, at a maximum point $p$ for $f_{s_n}$, we have
        \begin{eqnarray*}
            \lambda \exp\left(2f_{s_n}(p)/C_t \right) &\geq& s_n S^{t} (\omega)(p) + \lambda (1-s_n) \geq s_n (S^{t}(\omega)(p) -\lambda) +\lambda\\[5pt]
            &\geq& \min \left\{ \min_M \left(S^{t}(\omega)\right) ,\, \lambda  \right\} \; ,
        \end{eqnarray*}
        while at a minimum point $q$ for $f_{s_n}$, there holds
        \begin{align}
            \lambda \exp\left(2f_{s_n}(q)/C_t \right) \leq s_n S^{t} (\omega)(q) + \lambda (1-s_n) \leq \max_M S^{t} (\omega) +\lambda \, .
        \end{align}
        Hence the lemma is proved.
    \end{proof}
    Now it remains to prove the uniform equicontinuity of the functions $\{f_{s_n}\}$ in $\mathcal{C}^{2,\alpha}(M ; \mathbb{R})$. We define the elliptic operators
    $$ L_n f : = \Delta^{Ch}_\omega f + s_n S^{t} (\omega) + \lambda (1-s_n) \;. $$
    For the functions $f_{s_n}$ we get the equalities
    $$ L_n f_{s_n} = \lambda \exp\left(2f_{s_n}/C_t \right) \; .$$
    The estimate of Lemma \ref{lem: C zero unif estimate} gives a uniform $L^\infty$ control of the right-hand side $\lambda \exp\left(2f_{s_n}/C_t \right)$ of the equation and hence a uniform $L^p$ control of $L_n f_{s_n}$ for any $p\in(1,\infty)$. Then, by the Calderon-Zygmund inequality we can control the $p$-norm of the second-order derivatives by the $p$-norms of the function and its Laplacian; hence, iterating it twice, we get that $f_{s_n}\in W^{4,p}(M;\R)$ with uniform bound on the norms. Finally, we can use the Sobolev embedding taking $p$ large enough so that we find an a-priori $\mathcal{C}^{3}$ uniform bound on the solutions. Thus now we can apply the Ascoli-Arzel\`a theorem so that we get a subsequence (which we still call $\{f_{s_n}\}$) converging in $\mathcal{C}^{2,\alpha}(M;\R)$ to a function $f_{s_\infty}$. We can take the limit in the equation (\ref{appia}); in this way we see that $f_{s_\infty}$ is a solution of $\mathrm{GaYa} (s_\infty , \, f_{s_\infty} ) = 0$ as needed.

\medskip

    So far we achieved the existence of a $\mathcal{C}^{2,\alpha}$ solution $f$ to the Gauduchon-Yamabe equation, $\mathrm{GaYa}(1,f)=0$. Hence we have $f\in\mathcal{C}^{2,\alpha}$ such that
    $$ \Delta^{Ch}_\omega f=\lambda e^{2f/C_t} - S^t(\omega) \;.$$
    Notice that the right-hand side has the same regularity of $f$, hence the smooth regularity of the solution follows by the usual bootstrap argument via Schauder's estimates for elliptic operators. 

\medskip

    Now we have a smooth function $f$ solving 
    $ \Delta^{Ch}_\omega f=\lambda e^{2f/C_t} - S^t(\omega) $
    and we want to prove its uniqueness.
    
    Notice that by Proposition \ref{prop:gamma-sign} since we have $\lambda=\Gamma_M^t(\{\omega\})$, $e^f\omega$ must be in $\{\omega\}_1$; moreover, any other metric in $\{\omega\}_1$ with constant $\nabla^t$-scalar curvature must solve the same equation.

    Now suppose we have two conformal metrics $\omega_1=\exp(2f_1/C_t )\omega$ and $\omega_2=\exp(2f_2/C_t)\omega$ in $\{\omega\}_1$ with constant $\nabla^t$-scalar curvatures equal to $\lambda$. Hence we have the equations
    $$ \Delta^{Ch}_{\omega} f_1 + S^{t}(\omega) = \lambda \exp \left( 2f_1/C_t \right)
    \qquad \text{ and } \qquad
    \Delta^{Ch}_{\omega} f_2 + S^{t}(\omega) = \lambda \exp \left( 2f_2/C_t \right) \;. $$
    Taking the difference of these, we get the equation
    $$ \Delta^{Ch}_{\omega} (f_1-f_2) = \lambda ( \exp \left( 2f_1/C_t \right) - \exp \left( 2f_2/C_t \right))\;.$$
    At a first glance, we should distinguish the cases $C_t>0$ or $C_t<0$ for which we respectively have $\lambda<0$ and $\lambda>0$; however, in both cases at a maximum point $p$ for $f_1-f_2$, we find
    $f_1 (p) - f_2 (p) \;\leq0$ while at a minimum point $q$, we have
    $f_1 (q) - f_2 (q) \;\geq0$, proving that $f_1$ and $f_2$ coincide.
\end{proof}

\begin{rmk}
    In case $C_t\Gamma_M^t(\{\omega\}) > 0$, the maximum principle does not apply and the Gauduchon-Yamabe equation loses its good analytical properties. For the Chern connection, this case corresponds to having a positive Gauduchon degree and it is investigated in Section 5 of \cite{Chern-Yamabe} where some sufficient criteria for the existence of positive constant Chern scalar curvature metrics are found. Moreover, non-homogeneous examples of Hermitian metrics of positive constant Chern scalar curvature have been constructed in \cite{Mehdi0} and \cite{Angella Pediconi}.\\
    Similarly, it would be interesting to find new explicit examples of constant Bismut scalar curvature metrics as well as some sufficient (and, possibly, necessary) conditions which ensure the existence of metrics with negative constant scalar curvature for the Bismut connection.\\
    We also remark that the "critical" Gauduchon connection for which the constant $C_t$ vanishes are left out by these theorems. In particular, in complex dimension $2$ it happens for the Bismut connection, since $C_t=1+nt-t=1+t$. 
\end{rmk}
We then have the following result.
\begin{cor}\label{cor: bis-yam}
    Let $M$ be a compact complex manifold with $\dim_\C M\geq 3$ and Hermitian structure $(\omega,J)$.
    If $\Gamma_M^+(\{\omega\}) \geq 0$, then there exists a unique $\widetilde{\omega}\in\{\omega\}_1$ with constant Bismut scalar curvature. Moreover, $S^{+}(\widetilde{\omega})=\Gamma_M^+(\{\omega\})$.
\end{cor}
\begin{rmk}
    In \cite{Mehdi} the authors extended the results of Angella, Calamai, and Spotti on the Chern-Yamabe problem to the non-integrable case. Our results (Theorems \ref{thm: zero gau-yam} $\&$ \ref{thm:  gau-yam}) can also be extended to the non-integrable case.
\end{rmk}

\section{Constant Bismut Ricci curvature}\label{sec: Bismut Einstein}
We study the CYT equation on rank one toric bundles over Hermitian manifolds imitating the setting of \cite{Gran Gran Poon}. In that article, the authors derive useful formulas for the Bismut Ricci curvature of special metrics on principal toric bundles over compact K{\"a}hler manifolds. Using them they construct CYT structures on the manifolds $(k-1)(\mathbb{S}^2\times\mathbb{S}^4)\# k(\mathbb{S}^3\times\mathbb{S}^3)$ for all $k \geq 1$.\\

Given a Hermitian manifold $(X,\omega_X)$, consider a principal toric bundle $$\mathbb{S}^1 \times\mathbb{S}^1 \hookrightarrow M \xrightarrow{\pi} X$$ with characteristic classes of type $(1,1)$. 
We take a connection one form with values in the Lie algebra of $\mathbb{S}^1\times\mathbb{S}^1$ given by $(\theta_1,\theta_2)$ such that $d\theta_i=\pi^*\omega_i$, with $\omega_i$ $(1,1)$-forms on $X$. 
Once we fix a complex structure on the torus, $M$ inherits a complex structure from that of $X$, such that the projection map $\pi$ from $M$ to $X$ is holomorphic, see Lemma 1 of \cite{Gran Gran Poon} for details on this. 
We consider the Hermitian metrics on $M$ for which $\pi$ becomes a Riemannian submersion. 
These are all of the forms
$$\omega=\pi^*(\omega_X) + f\theta_1\wedge\theta_2 \;,$$ 
where $f$ is a positive function on $M$, which is constant along the fibers (thus we will usually think of it as a function on $X$). 
The Bismut Ricci form of $\omega$ is given by
\begin{equation}\label{eq: zero}
        Ric^+(\omega)=\pi^*\left(Ric^+(\omega_X)\right) - dd^* (f\theta_1\wedge\theta_2)\;.
\end{equation}
Indeed, from Proposition \ref{prop: relation Bismu Chern Ricci},
\begin{align*}
    Ric^+\omega&= Ric^{Ch}\omega - dd^*\omega = \pi^*Ric^{Ch}\omega_X - dd^*\omega = \pi^*\left(Ric^+\omega_X + dd^*\omega_X\right) - dd^*\omega\\
    &= \pi^*\left(Ric^+\omega_X\right) + dd^*(\pi^*\omega_X)- dd^*\omega = \pi^*Ric^+(\omega_X) - dd^* (f\theta_1\wedge\theta_2)
\end{align*}
where the second equality comes from the following lemma.
\begin{lem*}[Lemma 3 of \cite{Gran Gran Poon}]
    Let $Ric^{Ch}\omega$ and $Ric^{Ch}\omega_X$ be the Ricci forms of the Chern connections on $(M,\omega)$ and $(X,\omega_X)$ respectively. Then $Ric^{Ch}\omega=\pi^*(Ric^{Ch}\omega_X)$.
\end{lem*} 
We work on a Hermitian frame $\{e_1,\ldots,e_{2n},t_1,t_2\}$ on an open subset of $M$ which comes from a local Hermitian frame $\{e_1,\ldots,e_{2n}\}$ on an open subset of $X$ extended so that the vector fields $t_1, t_2$ are dual to the $1$-forms $\theta_1,\theta_2$.
\begin{lem}
    The following equations hold:
    \begin{itemize}
        \item $[t_i,e_j]=0$ for any $i=1,2$ and $j=1,\ldots,2n$;
        \item $\theta_i(\sum_j [e_{2j-1},e_{2j}])=-tr_{\omega_X}(\omega_i)$ for $i=1,2$.
    \end{itemize}
\end{lem}
\begin{proof}
    We derive these equations from the conditions $d\theta_i=\pi^*\omega_i$ ($i=1,2$).\\
    First of all, since the Lie brackets are $\pi$-related, i.e. $\pi_*[u,v]=[\pi_*u,\pi_*v]$ for any smooth vector fields $u,v$, we have that $[t_i,e_j]$ must be vertical. However, 
    $$ \theta_k([t_i,e_j])=-d\theta_k(t_i,e_j)=-\pi^*\omega_k(t_i,e_j)=0 $$
    We similarly obtain the second equation, indeed
    $$ \theta_i([e_{2j-1},e_{2j}])=-d\theta_i(e_{2j-1},e_{2j})=-\omega_i(e_{2j-1},e_{2j})$$
    thus
    $$ \theta_i\left(\sum_j [e_{2j-1},e_{2j}]\right)=-\sum_j \omega_i(e_{2j-1},e_{2j})=-tr_{\omega_X}\,\omega_i\;. $$
\end{proof}
We now compute $dd^*\widehat{\omega}$ where $\widehat{\omega}:=f\theta_1\wedge\theta_2$. First of all, recall that the co-differential of a tensor could be expressed in terms of the contraction of the Levi--Civita connection as
\begin{equation*}
    d^*\widehat{\omega}=-\sum_{j=1}^{2n} \nabla^{LC}_{e_j}\widehat{\omega}(e_j,\cdot)-\sum_{i=1,2} \nabla^{LC}_{t_i}\widehat{\omega}(t_i,\cdot)\;,
\end{equation*}
moreover for any smooth vector fields $u,v,w$ on a Hermitian manifold
\begin{equation*}
    -2\left(\nabla^{LC}_u \widehat{\omega}(v,w)\right)=d\widehat{\omega}(u,Jv,Jw)-d\widehat{\omega}(u,v,w)\;,
\end{equation*}
hence we have
\begin{align}\label{eq: dd^*omega}
    dd^*\widehat{\omega}=&\, d\left( \sum_j d\widehat{\omega}(J\cdot,e_{2j-1},e_{2j}) + d\widehat{\omega}(J\cdot,t_1,t_2)\right) \nonumber \\
    =&\, d\left(\sum_j  d\widehat{\omega}\left(t_2,e_{2j-1},e_{2j}\right)\theta_1 -\sum_j d\widehat{\omega}\left(t_1,e_{2j-1},e_{2j}\right)\theta_2 + d\widehat{\omega}(J\cdot,t_1,t_2)\right) \nonumber \\
    =&\, d\left( -\sum_j \widehat{\omega}\left([e_{2j-1},e_{2j}],t_2\right) \theta_1 + \sum_j \widehat{\omega}\left([e_{2j-1},e_{2j}],t_1\right) \theta_2 + d\widehat{\omega}(J\cdot,t_1,t_2)\right) \nonumber \\
    =&\, d\left( -f\left(\theta_1 \left(\sum_j [e_{2j-1},e_{2j}]\right)\theta_1 + \theta_2 \left(\sum_j [e_{2j-1},e_{2j}]\right)\theta_2 \right) + \left(\left((J e_j) \right)f\right)e^j \right) \nonumber\\
    =&\, d\left( f\left(tr\,\omega_1\,\theta_1 + tr\,\omega_2\,\theta_2 \right) + \left((J e_j) f\right)e^j \right) \nonumber\\
    =&\, df\wedge(tr\,\omega_1\,\theta_1 + tr\,\omega_2\,\theta_2) + fd(tr\,\omega_1\,\theta_1 + tr\,\omega_2\,\theta_2) + d\left(\left((J e_j) f\right)e^j\right) \nonumber \\
    =&\, \left( e_i\, (J e_j) f - e_j \, (J e_i) f\right)e^i\wedge
     e^j + f(tr\,\omega_1\,\pi^*\omega_1 + tr\,\omega_2\,\pi^*\omega_2) \\
    & + \left[  e_j (f \, tr\,\omega_i) - t_i\, (J e_j) f \right] e^j\wedge \theta^i + \left[ t_1(f\, tr\,\omega_2) - t_2(f\, tr\,\omega_1) \right] \theta_1\wedge\theta_2 \nonumber
\end{align}
where we used the Einstein notation and dropped the subscript $\omega_X$ on the traces $tr_{\omega_X}\,\omega_i$ for convenience. Since $f$ is constant along the fibers we obtain
\begin{equation}\label{eq: dd^*omega const fib}
    dd^*\widehat{\omega}= \pi^*dd^cf + f(tr\,\omega_1\,\pi^*\omega_1 + tr\,\omega_2\,\pi^*\omega_2) + \left[  e_j (f \, tr\,\omega_i) \right] e^j\wedge \theta^i \;.
\end{equation}
From this identity and equation (\ref{eq: zero}) we get the following result.
\begin{prop}\label{prop: Bismut Ricci flat}
    On the total space $M$ of a principal toric bundle of rank one $\mathbb{S}^1 \times\mathbb{S}^1 \hookrightarrow M \xrightarrow{\pi} X$ over a Hermitian manifold $(X,\omega_X)$ with connection one forms $(\theta_1,\theta_2)$, the metric $\omega=\pi^*(\omega_X) + f\theta_1\wedge\theta_2 \;,$ defines a Calabi--Yau with torsion structure if
    \begin{equation*}
        \begin{cases}
            Ric^+(\omega_X) = dd^c f +(c_1\omega_1 + c_2\omega_2)\\
            f\,tr_{\omega_X}\,\omega_i = c_i
        \end{cases}
    \end{equation*}
\end{prop}
Notice that if $tr_{\omega_X}\,\omega_i$ vanishes at some point then it must vanish everywhere since $f\,tr_{\omega_X}\,\omega_i$ is supposed to be a constant function and $f>0$. 
We now analyze the case of $tr_{\omega_X}\,\omega_1$, and $tr_{\omega_X}\,\omega_2$ simultaneously zero, which is completely understood from the following Lemma. 
\begin{lem}[Lemma 6 of \cite{Gran Gran Poon}]\label{lemma: conformally bismut ricci flat}
    Suppose that the Ricci form of the Bismut connection of a Hermitian metric $g_M$ is $\partial\overline{\partial}$-exact on a manifold $M$ of dimension greater than two. Then the metric $g_M$ is conformally a CYT structure. In other words, there exists a conformal change of $g_M$ such that the Ricci form of the induced Bismut connection vanishes.
\end{lem}
\begin{proof}
    We recall the argument of the proof in \cite{Gran Gran Poon}. The result directly comes from the formula for the conformal change of the Ricci curvature form of the Gauduchon connections. Indeed, setting $t=-1$ in (\ref{eq: conformal Ricci Gauduchon}) we obtain
    \begin{equation*}
        \left(Ric^+(e^f\omega)\right)^{1,1}=\left(Ric^+(\omega)\right)^{1,1} +(n-2)dd^cf \;.
    \end{equation*}
    Then it is sufficient to notice that the $(2,0)$ and the $(0,2)$ components of the Bismut Ricci tensor are invariant for conformal changes and are zero by hypothesis. 
\end{proof}
Without loss of generality, we may consider a trivial $\S^1\times\S^1$-principal bundle over $(X,\omega_X)$ since the computations are the same. In this case, we have the following result.
\begin{prop}
    Given a compact Hermitian manifold $(X,\omega_X)$ such that the Bismut Ricci curvature is $dd^c$-exact, i.e. $Ric^+ (\omega_X)=\sqrt{-1}\partial\overline{\partial}f$, then $M:= \mathbb{S}^1 \times\mathbb{S}^1 \times X$, equipped with the induced complex structure, admits a CYT metric.
\end{prop}
\begin{proof}
    We can suppose that $f$ is positive since it is defined up to constants on a compact manifold. Thus we can define the metric on $M$ as usual: $ \omega = \pi^*(\omega_X) + f \theta_1 \wedge \theta_2$, where $\pi$ is the trivial projection of $M$ onto $X$ and $\theta_1,\theta_2$ are dual to the coordinates $t_1,t_2$ on the fibers.
    We hence have that 
    $$ Ric^+(\pi^*(\omega_X)+f \theta_1 \wedge \theta_2)= \pi^*(Ric^+(\omega_X)) - dd^*(f \theta_1 \wedge \theta_2) =0$$
    since from (\ref{eq: dd^*omega const fib}) we get
    \begin{align*}
        dd^* ( f \theta_1 \wedge \theta_2) &= \pi^*(dd^cf)\;.
    \end{align*}
\end{proof}
From this result, using the transformation law of the Bismut Ricci form under a conformal change (\ref{eq: conformal Ricci Gauduchon}), we directly get the following corollary.
\begin{cor}\label{cor: lots of CYT}
    Given a complex manifold $(X,J)$ with a CYT metric $g_X$, for any positive function $f>0$, we can construct a CYT structure on $\mathbb{S}^1 \times\mathbb{S}^1 \times X$ by taking the submersion metric $\omega=\pi^*(e^f\omega_X) + (n-2)f\,\theta_1\wedge\theta_2$.
\end{cor}

\subsection{Class $\mathcal{C}$ manifolds} 
In \cite{Podesta} the author defined Class $\mathcal{C}$ manifolds as homogeneous manifolds $M = G/L$, where $G = G_1 \times G_2$ are compact simply connected simple Lie groups $G_1, G_2$ and $L$ is a connected closed subgroup of $G$. 
There should also exist two irreducible compact Hermitian symmetric spaces $G_1/H_1, G_2/H_2$ so that the subgroups $H_i$ are of the form $H_i =\, <Z_i>\cdot\, L_i$ for $i = 1,2$ and $L = L_1 \times L_2$.
Therefore, we have the following setting
\begin{equation*}
    \begin{tikzcd}
         & & & G_1/H_1\\ 
        \S^1\times\S^1 \arrow[r,hook] & (G_1/L_1)\times(G_2/L_2) \arrow[r,"\phi_1\times\phi_2"] & (G_1/H_1)\times(G_2/H_2) \arrow[ur, "\pi_1"] \arrow[dr, swap, "\pi_2"] \\
         & & & G_2/H_2
    \end{tikzcd}
\end{equation*}
where the $\phi_i$ are the Tits fibrations given by
$$ \phi_i : G_i/L_i \rightarrow G_i/H_i :\, g\cdot L_i \mapsto g\cdot H_i  \;.$$
$M$ is then the product of two manifolds
$ M=(G_1/L_1)\times(G_2/L_2) $
where $G_1/L_1$ and $G_2/L_2$ are M-manifolds as defined in \cite{Podesta}, meaning that $L_1$ and $L_2$ are the semi-simple part of the centralizer of some torus, $H_i = C_{G_i} (<Z_i>)$. Then by Theorem C in \cite{Wang} the manifolds $G_i/H_i$ are also simply connected, hence they are generalized flag manifolds. 
In \cite{Matsushima}, it is proved that any generalized flag manifold can be endowed with an invariant K{\"a}hler–Einstein Fano metric which is unique (up to homothety) once we fix the invariant complex structure on it.
Moreover, the left-invariant complex structures on $M$ are all given by choosing left-invariant complex structures on the symmetric spaces $G_1/H_1, G_2/H_2$ and on the torus $\S^1\times\S^1$. Once we fix the complex structures on the symmetric spaces, the standard complex structure on $M$ is that given by choosing $I(Z_1)=Z_2$.
By exploiting the structure of class $\mathcal{C}$ manifolds we can explicitly construct submersion metrics on them which are CYT, namely, we prove the following theorem.
\begin{thm}\label{thm: CYT class C}
    Take a class $\mathcal{C}$ manifold as in \cite{Podesta}, that is a product $M=M_1\times M_2$ of M-manifolds whose fibers through the Tits fibrations $\phi_i$ over two generalized flag manifolds $X_1=G_1/H_1$ and $X_2=G_2/H_2$ with $\S^1$-fibers, and equip it with a standard complex structure. 
    Set $\omega_i$ the unique invariant K{\"a}hler--Einstein metrics on $X_i$ with Einstein constants $n_i=\dim(X_i)$. Then the metric on $M$ given by
    $$ \omega= \phi_1^*(\omega_1) + \phi_2^*(\omega_2) + \theta_1\wedge\theta_2 \;,$$
    defines a CYT structure on $M$. Here $\theta_1$ and $\theta_2$ are the connections one forms on the fiber bundles such that $d\theta_i=\phi^*_i\,\omega_i$, 
\end{thm}
\begin{proof}
    The metric on the base space $X=X_1\times X_2$ is $\omega_X = \omega_1 + \omega_2$. Then the metric $\omega$ satisfies
    \begin{equation*}
        \begin{cases}
            Ric^+(\omega_X) = dd^c(1) + n_1\omega_1 + n_2\omega_2\\
            tr_{\omega_X}\,\omega_i=n_i \quad \text{ for } i=1,2
        \end{cases}
    \end{equation*}
    hence it is Bismut Ricci flat by Proposition \ref{prop: Bismut Ricci flat}.
    Thus, we only need to check that the Tits fibrations represent the $U(1)$-principal bundles over $G_i/H_i$ with curvature $\omega_i \in c_1(G_i/H_i)$ chosen to be the unique K{\"a}hler--Einstein metrics on $X_i$. We know that the isomorphism classes of principal $U(1)$-bundles over a manifold $X$ are parametrized by its cohomology group $H^2(X)$; moreover, we can extract the following piece from the exact sequence in the cohomology of the Tits fibration: 
    $$ \R\cong H^1(\S^1)\xrightarrow{\delta} H^2(G/H)\xrightarrow{\phi^*}H^2(G/L) =0 \;,$$
    where the last term vanishes since the M-manifolds have zero second Betti number by Theorem D in \cite{Wang}. 
    Thus, on the $U(1)$-principal bundles on the $G_i/H_i$ given by the Tits fibrations we can always find connections one forms $\theta_i$ with curvature in $c_1(G_i/H_i)$.
\end{proof}

The existence of CYT Hermitian structures on the class $\mathcal{C}$ manifolds can also be derived by Theorem 3 in \cite{Gran}. 
Indeed, the metric $-B(\cdot,\cdot)$ given by the negative of the Killing form of $G$ is Hermitian with respect to the standard complex structures on $M$.  
To see this, consider the decomposition of the Lie algebra $\g$ of $G$ as
$$ \g=\g_1+\g_2 = (\m_1 + \l_1) + (\m_2 +\l_2) = (\n_1 + \R\,Z_1 + \l_1) + (\n_2 + \R\,Z_2 + \l_2) \;. $$
Here, $\g_1$ and $\g_2$ are the Lie algebras of $G_1$ and $G_2$ respectively, while the $\l_i$'s are the Lie algebras of the $L_i$'s; moreover, the Lie algebras $\h_i$'s of the $H_i$'s satisfy $\h_i = \l_i + \R\,Z_i$ for $i=1,2$.
The Killing form $B$ is Hermitian on $\n$, moreover, the tori $\left<Z_1\right>$ and $\left<Z_2\right>$ are orthogonal to the $\n_i$'s as well as one to each other. 
It only remains to verify that $B$ is Hermitian on $\t$, that is $B(Z_1,Z_1)=B(Z_2,Z_2)$.

The CYT metrics constructed above can be characterized as the unique CYT metrics among the homogeneous ones. 
Namely, we prove the following result.

\begin{thm}\label{thm: CYT uniqueness}
    Take a class $\mathcal{C}$ manifold $M$ as in Theorem \ref{thm: CYT class C}. Suppose that none of the $X_i$'s is $\mbox{SO}(k+2)/\mbox{SO}(2)\times\mbox{SO}(k)$ for $k\geq 3$, then the metric
    $$ \omega= \phi_1^*(\omega_1) + \phi_2^*(\omega_2) + \theta_1\wedge\theta_2 \;,$$
    constructed in Theorem \ref{thm: CYT class C}, is the unique (up to homothety) homogeneous CYT metric on $M$.
\end{thm}
\begin{proof}
    First of all, we verify that the homogeneous metrics on $M$ make $\phi_1\times\phi_2$ a Riemannian submersion. 
    Indeed, with the same notations as above, a $G$-invariant Hermitian metric $g'$ on $M$, can be seen as an $\ad\left(\l_1+\l_2\right)$-invariant Hermitian inner product on $\m_1+\m_2$. 
    As the $\l_i$'s are not trivial, $\l = \l_1+\l_2$ acts non-trivially on $\n=\n_1+\n_2$ and trivially on $\t$, therefore $g'(\t, \n) = 0$.
    Moreover, the $\ad(\l)$-modules $\n_i$ are mutually non-equivalent, hence $g'(\n_1, \n_2 ) = 0$. 
    Since we are avoiding the special case of $\g_i = \mathfrak{so}(n + 2)$ and $\h_i = \mathfrak{so}(2)+\mathfrak{so}(n)$, for $ n \geq 3$ the $\n_i$'s are $\l_i$-irreducible. Hence, the Schur Lemma implies that $g'$ on $\n_i\times\n_i$ restricts to a multiple ($\lambda_i\in\R_+$) of the Killing form on $G_i$, i.e.
    $$ g'_{|_{\n_i\times\n_i}} = -\lambda_i (B_i)_|\;.$$ 
    In other words, the homogeneous metrics on $M$ are all of the types
    $$ \omega' = \lambda_1\phi_1^*(\omega_1) + \lambda_2\phi_2^*(\omega_2) + \lambda \theta_1\wedge\theta_2 \;,$$
    and by Proposition \ref{prop: Bismut Ricci flat} any homogeneous CYT metric $g'$ have to satisfy
    \begin{equation*}
        \begin{cases}
            Ric^+(\lambda_1\omega_1 + \lambda_2\omega_2) = n_1\omega_1 + n_2\omega_2\\
            \lambda tr_{(\lambda_1\omega_1+\lambda_2\omega_2)}\,\omega_i=n_i \quad \text{ for } i=1,2
        \end{cases}
    \end{equation*}
    However, for $i=1,2$,
    $$ n_i = \lambda tr_{(\lambda_1\omega_1+\lambda_2\omega_2)}\,\omega_i=\frac{\lambda}{\lambda_i}tr_{(\omega_1+\omega_2)}\,\omega_i=\frac{\lambda}{\lambda_i}n_i \;,$$
    proving that $\lambda=\lambda_1=\lambda_2$, and hence $g'$ is a positive multiple of $g$.
\end{proof}

\medskip

Theorem \ref{thm: CYT class C} and Theorem \ref{thm: CYT uniqueness} apply to give unique homogeneous CYT metrics on the Calabi--Eckmann manifolds when they are equipped with their standard complex structures. 
Indeed, these are the total spaces $M_{n_1,n_2}\cong\mathbb{S}^{2n_1+1}\times\mathbb{S}^{2n_2+1}$ of rank one toric bundles over the product of complex projective spaces $\mathbb{C}\mathbb{P}^{n_1}\times\mathbb{C}\mathbb{P}^{n_2}$. 
As class $\mathcal{C}$ manifolds, they are given by taking $G_i=\mbox{SU}(n_i+1),\, L_i = \mbox{SU}(n_i),$ and $H_i = \mbox{SU}(n_i) \times \mbox{U}(1)$. 
The Tits fibrations agree with the Hopf fibrations $\mathbb{S}^1 \hookrightarrow  \mathbb{S}^{2n_i+1}\xrightarrow{\phi_i} \mathbb{C}\mathbb{P}^{n_i}$.
Then the standard Calabi--Eckmann structure $J$ on 
$$ TM_{n_1,n_2}= H_{\mathbb{C}\mathbb{P}^{n_1}}\oplus \left<Z_1, Z_2\right>\oplus H_{\mathbb{C}\mathbb{P}^{n_2}}  $$
is $J=J_{\mathbb{C}\mathbb{P}^{n_1}}\oplus I \oplus J_{\mathbb{C}\mathbb{P}^{n_2}}$ where $J_{\mathbb{C}\mathbb{P}^{n_1}}$ and $J_{\mathbb{C}\mathbb{P}^{n_2}}$ are the complex structures of $\mathbb{C}\mathbb{P}^{n_1}$ and $\mathbb{C}\mathbb{P}^{n_2}$ respectively pulled-back on the horizontal spaces $H_{\mathbb{C}\mathbb{P}^{n_1}}$ and $H_{\mathbb{C}\mathbb{P}^{n_2}}$, and $I\left(Z_1\right)=Z_2$.

\begin{cor}\label{cor: CYT on CE}
    Given a Calabi--Eckmann manifold $M_{n_1,n_2}$ equipped with its standard complex structure. 
    Consider the Fubini--Study metrics $\omega_i$'s on the complex projective spaces $\C\P^{n_i}$'s with Einstein constants $n_i$'s, and set $\theta_1$ and $\theta_2$ the connections one forms on the fiber bundles such that $d\theta_i=\phi^*_i\,\omega_i$ for $i=1,2$. 
    Then, the metric 
    $$ \omega= \phi_1^*(\omega_1) + \phi_2^*(\omega_2) + \theta_1\wedge\theta_2 \;,$$
    is a CYT metric on $M_{n_1,n_2}$; moreover, it is the unique homogeneous CYT metric on it.
\end{cor}

We remark that in \cite{me} a complete description of the Bismut curvature tensor for the homogeneous metrics on the Hopf manifolds was given. In particular, given a homogeneous metrics $g(\alpha,\beta)$ on an $n$-dimensional Hopf manifold, it holds (in the standard local holomorphic coordinates $\{z_i\}$),
$$ Ric^+ (g(\alpha,\beta)) = \left(2-n + 2\frac{\beta}{\alpha}(1-n)\right)\left(\frac{\delta_{ij}}{|z|^2}-\frac{\overline{z}_i z_j}{|z|^4}\right) \, dz^i\wedge d\bar z^j \;. $$
This identically vanishes if and only if the ratio $\frac{\beta}{\alpha}$ equals $\frac{2-n}{2n-2}$, giving explicit CYT metrics on any Hopf manifold $\S^1\times\S^{2n+1}$, unique among the homogeneous ones. 

\subsection{Bismut Hermitian Einstein}
The {\em pluriclosed flow} is a parabolic flow of metrics in the family of Hermitian curvature flows introduced by Streets and Tian in \cite{Streets Tian 4}. 
It has the property of preserving the pluriclosed condition, that is $\p\bar\p\omega=0$.
As a matter of fact, given a pluriclosed metric $\omega_0$, it evolves as
\[
\begin{cases}
    \frac{\p}{\p t}\omega = -\left(Ric^+(\omega)\right)^{1,1}\\
    \omega(0) = \omega_0
\end{cases}
\]
The static points of the pluriclosed flow are pluriclosed metrics $\omega$ which satisfy
\begin{equation}\label{1,1 Bismut Einstein}
    \left(Ric^+(\omega)\right)^{1,1} = \lambda \omega \, ,\quad\lambda\in\R\;,
\end{equation}
and are called in the literature {\em Bismut Hermitian Einstein metrics}. 
We distinguish the case of $\lambda=0$ from that of $\lambda\neq 0$. 
The only known examples of non-K{\"a}hler pluriclosed metrics which also satisfy $(Ric^+\omega)^{1,1}=0$ are that given by the {\em Bismut flat} structures. 
These are Hermitian structures such that the whole Bismut curvature tensor $R^+$ vanishes and Theorem 9 in \cite{Angella et al} ensures that their Hermitian metrics are pluriclosed.
Particular examples are given by the standard Calabi--Eckmann structures on the Hopf surface $\S^1\times\S^3$ and the Calabi--Eckmann threefold $\S^3\times\S^3$ which are known to be Bismut flat.
In \cite{Garcia Streets} the authors asked if the other Calabi--Eckmann manifolds admit such special Hermitian structures. 
The answer is negative and it comes from the fact that, for cohomological reasons, $\S^1\times\S^1$, $\S^1\times\S^3$ and $\S^3\times\S^3$ are the only Calabi--Eckmann manifolds which can admit a pluriclosed structure, see Example 5.17 in \cite{Cavalcanti}. 
On the other hand, when equipped with the standard Calabi--Eckmann complex structure, these manifolds can be equipped with metrics such that $(Ric^+\omega)^{1,1}$ vanishes. Hence we have the following picture:
\begin{itemize}
    \item $\mathbb{S}^1\times\mathbb{S}^1$ has a flat K{\"a}hler metric;
    \item $\mathbb{S}^3\times\mathbb{S}^1$ has a Bismut flat, hence pluriclosed, metric;
    \item $\mathbb{S}^3\times\mathbb{S}^3$ has a Bismut flat, hence pluriclosed, metric;
    \item $\S^{2n+1}\times\S^{2m+1}$ with $n\geq2$, $m\geq 0$ have Bismut Ricci flat metrics (by Corollary \ref{cor: CYT on CE}) which are not pluriclosed. 
\end{itemize}
On the other hand, when $\lambda\neq 0$ there are restrictions that suggest that the equation \eqref{1,1 Bismut Einstein} should imply that the Hermitian structure is K{\"a}hler (see \cite{Garcia Streets}, page 172). 
Some of these can be found in \cite{Streets solitons} (e.g. Proposition 3.5), where the author classifies solitons of the Pluriclosed Flow. 

\begin{prop*}[Proposition 3.5 of \cite{Streets solitons}]\label{soliton street}
    Let $(M^{2n},J)$ be a compact K{\"a}hler manifold, and suppose $(g, f)$ is a pluriclosed steady or shrinking soliton on $M$. Then $(g, f)$ is a K{\"a}hler-Ricci soliton.
\end{prop*}

Here we show that on $\S^1\times \S^1$-principal bundles over Hermitian manifolds there are no metrics satisfying \eqref{1,1 Bismut Einstein} with $\lambda\neq0$.
We take a slightly more general setting than that of the previous section, namely, we equip the total space $M$ of the rank one toric fibration over $(X,\omega_X)$ with a metric $ \omega = \pi^*(\omega_X) + f \theta_1 \wedge \theta_2 $ where $f$ is any positive function on $M$. Using the computation of \eqref{eq: dd^*omega} we prove the following result.
\begin{prop}
    Given a principal toric bundle, $\mathbb{S}^1 \times\mathbb{S}^1 \hookrightarrow M \xrightarrow{\pi} X$, over a Hermitian manifold $(X,\omega_X)$, with connection one forms $\theta_1,\theta_2$ such that $d\theta_i=\pi^*\omega_i$ for $(1,1)$-forms $\omega_i$, there are no Hermitian metric of type $\omega=\pi^*(\omega_X) + f\theta_1\wedge\theta_2$, where $f$ is a positive function on $M$, which satisfy the equation
    $$ \left(Ric^+(\omega)\right)^{1,1}=\lambda\omega $$
    for $\lambda\neq 0$.
\end{prop}
\begin{proof}
    Thanks to (\ref{eq: dd^*omega}) the Bismut Einstein problem (\ref{1,1 Bismut Einstein}) in this setting reduces to solve \begin{equation*}
        \begin{cases}
            \pi^*\left(Ric^+(\omega_X)\right)^{1,1} = \lambda\pi^*(\omega_X) + \left( e_i (J e_j)f - e_j (J e_i) f\right)e^i\wedge e^j + f(tr_{\omega_X}\omega_1\pi^*\omega_1 + tr_{\omega_X}\omega_2\pi^*\omega_2)\\
            tr_{\omega_X}\omega_1\,t_2f  - tr_{\omega_X}\omega_2\,t_1f = \lambda f\\
            \left(\left(  e_j (f \, tr_{\omega_X}\omega_i) - t_i(J e_j) f \right) \theta^i\wedge e^j\right)^{1,1} =0
        \end{cases}
    \end{equation*}
    In particular, $f$ has to verify
    \begin{equation*}\label{1,1 Bismut Einstein tris}
        tr_{\omega_X}\omega_1\,t_2f  - tr_{\omega_X}\omega_2\,t_1f = \lambda f\;.
    \end{equation*}
    Since the fibers are compact, if we fix one of them, there should be a critical point for $f$ on it. At this point, both $t_1 f$ and $t_2 f$ vanish giving a contradiction with the above equality, since $f>0$.
\end{proof}

\section*{Acknowledgements}
I would like to thank my advisor Daniele Angella, and Francesco Pediconi for many helpful suggestions and their constant support and encouragement. 
I am also grateful to professors Simone Calamai, David Petrecca, and Cristiano Spotti for useful clues and discussions. 
Many thanks also to the anonymous Referees for their useful comments and suggestions.


\end{document}